\documentclass{amsart}

\usepackage[latin1]{inputenc}
\usepackage[T1]{fontenc}
\usepackage{enumerate}
\usepackage{mathtools}
\usepackage{amssymb}
\usepackage{mathrsfs}
\usepackage{mathabx}

\newtheorem{theorem}{Theorem}[section]

\newtheorem{cor}[theorem]{Corollary}
\newtheorem{prop}[theorem]{Proposition}

\theoremstyle{definition}
\newtheorem{definition}[theorem]{Definition}

\newtheorem{problem}{Problem}

\theoremstyle{remark}
\newtheorem{remark}[theorem]{Remark}

\numberwithin{equation}{section}

\newcommand{\N}{\mathbb N}

\title[Between $\mathcal{U}$-frequent hypercyclicity and frequent hypercyclicity]{A bridge between $\mathcal{U}$-frequent hypercyclicity and frequent hypercyclicity}
\author[Q. Menet]{Quentin Menet}

\address{Quentin Menet, Univ. Artois, EA 2462, Laboratoire de Mathématiques de Lens (LML), F-62300 Lens, France}
\email{quentin.menet@univ-artois.fr}
\thanks{The author was supported by the grant ANR-17-CE40-0021 of the French National 
Research Agency ANR (project Front)}
\subjclass[2010]{47A16}
\keywords{}

\allowdisplaybreaks[3]
\begin{document}
\begin{abstract}
Given $\mathcal{A}$ the family of weights $a=(a_n)_n$ decreasing to $0$ such that the series $\sum_{n=0}^{\infty} a_n$ diverges, we show that the supremum on $\mathcal{A}$ of lower weighted densities coincides with the unweighted upper density and that the infimum on $\mathcal{A}$ of upper weighted densities coincides with the unweighted lower density. We then investigate the notions of $\mathcal{U}$-frequent hypercyclicity and frequent hypercyclicity associated to these weighted densities. We show that there exists an operator which is $\mathcal{U}$-frequently hypercyclic for each weight in $\mathcal{A}$ but not frequently hypercyclic, although the set of frequently hypercyclic vectors always coincides with the intersection of sets of $\mathcal{U}$-frequently hypercyclic vectors for each weight in $\mathcal{A}$.
\end{abstract}
\maketitle

Let $a=(a_n)_{n\ge 0}$ be a sequence of positive real numbers such that $\sum_{n=0}^{\infty}a_n=\infty$. We define the upper density $\overline{d_a}$ and the lower density $\underline{d_a}$ by
\[\overline{d_a}(I)=\limsup_{N\to \infty} \frac{\sum_{n\in [0,N]\cap I}a_n}{\sum_{n=0}^{N}a_n} \quad \text{and}\quad \underline{d_a}(I)=\liminf_{N\to \infty} \frac{\sum_{n\in [0,N]\cap I}a_n}{\sum_{n=0}^{N}a_n}.\]

In particular, if $a_n=1$ for every $n\ge 0$, we get the upper and lower unweighted densities. We know thanks to Ernst and Mouze \cite{Ernst} that if $(a_n/b_n)$ decreases to $0$ then for every $I\subset \mathbb{N}$,\[\underline{d_b}(I)\le \underline{d_a}(I)\le \overline{d_a}(I)\le \overline{d_b}(I).\]
In their paper, Ernst and Mouze \cite{Ernst} are interested in the lower weighted densities $\underline{d_a}$ smaller than the lower unweighted density $\underline{d}$. In this paper, we will be interested in the weighted densities between the unweighted densities $\underline{d}$ and $\overline{d}$. To this end, we will focus on the densities $\underline{d_a}$ and $\overline{d_a}$ where $a$ is a decreasing sequence tending to $0$ such that $\sum_{n=0}^{\infty}a_n=\infty$. We denote by $\mathcal{A}$ the set of these sequences. Our interest in these densities comes from the study of two important notions in linear dynamics: $\mathcal{U}$-frequent hypercyclicity and frequent hypercyclicity.

Given $X$ a separable infinite-dimensional Fréchet space and $T$ a continuous linear operator on $X$, the orbit of a vector $x$ in $X$ under the action of $T$ is given by the set $\text{Orb}(x,T)=\{T^nx:n\ge 0\}$.
Linear dynamics is the theory investigating the properties of these orbits (see \cite{Bbook} and \cite{Kbook} for more information). For instance, we say that $T$ is hypercyclic if there exists $x\in X$ such that $\text{Orb}(x,T)$ is dense, or equivalently, such that for each non-empty open set $U$, the set $N(x,U)=\{n\ge 0:T^nx\in U\}$ is non-empty.

If the orbit of a vector visits each non-empty open set then this orbit visits infinitely often each non-empty open set and we can investigate the frequency of these visits. This study has been started by Bayart and Grivaux~\cite{2Bayart0, 2Bayart} via the notion of frequently hypercyclic operators. An operator $T$ is said to be frequently hypercyclic (FHC) if there exists a vector $x$ (called a frequently hypercyclic vector) such that for each non-empty open set $U$, we have $\underline{d}(N(x,U))>0$. In the same way, Shkarin~\cite{Shkarin} introduced the notion of $\mathcal{U}$-frequently hypercyclic operators (UFHC) by replacing the lower density by the upper density. These two notions of hypercyclicity have interesting differences. For instance, the set $UFHC(T)$ of $\mathcal{U}$-frequently hypercyclic vectors is either empty or residual~\cite{BayartR, GrivM, Moot} but the set $FHC(T)$ of frequently hypercyclic vectors is always meager~\cite{BayartR}. For this reason, we try to create a bridge between these two notions in the hope to better understand their differences and their limits.

\begin{definition}
Let $X$ be a separable infinite-dimensional Fréchet space, $T$ a continuous linear operator on $X$ and $a\in \mathcal{A}$.
\begin{itemize}
\item $T$ is said to be frequently hypercyclic with respect to $a$ (in short, $FHC_a$) if there exists $x\in X$ such that for every non-empty open set $U$, we have $\underline{d_a}(N(x,U))>0$.
\item $T$ is said to be $\mathcal{U}$-frequently hypercyclic with respect to $a$ (in short, $UFHC_a$) if there exists $x\in X$ such that for every non-empty open set $U$, we have $\overline{d_a}(N(x,U))>0$.
\end{itemize}
We denote by $FHC_a(T)$ (resp. $UFHC_a(T)$) the set of vectors which are frequently hypercyclic (resp $\mathcal{U}$-frequently hypercyclic) with respect to $a$ for $T$.
\end{definition}
 In view of the result of Ernst and Mouze \cite{Ernst}, we have for every $a\in \mathcal{A}$
\[FHC\quad\Rightarrow\quad FHC_a \quad\Rightarrow\quad UFHC_a \quad\Rightarrow\quad UFHC.\]
We can therefore wonder if the notions of $FHC_a$ and $UFHC_a$ allow to create a bridge between UFHC and FHC and which of these three possibilities is correct:
\begin{enumerate}
\item $\sup_{a\in \mathcal{A}}\underline{d}_{a}<\inf_{a\in \mathcal{A}}\overline{d}_{a}$: It means that there is a gap between the weighted lower and upper densities and we would have to investigate this gap.
\item $\sup_{a\in \mathcal{A}}\underline{d}_{a}=\inf_{a\in \mathcal{A}}\overline{d}_{a}$: The weighted lower and upper densities share a limit density which would be particularly interesting.
\item $\sup_{a\in \mathcal{A}}\underline{d}_{a}>\inf_{a\in \mathcal{A}}\overline{d}_{a}$: We can get a bridge between UFHC and FHC by using weighted lower densities and weighted upper densities.
\end{enumerate}

We will show in Section~\ref{Sec1} that the correct possibility is the third. In fact, we can show that $\sup_{a\in \mathcal{A}}\underline{d}_{a}=\overline{d}$ and that $\inf_{a\in \mathcal{A}}\overline{d}_{a}=\underline{d}$. These equalities will result in the following ones for $\mathcal{U}$-frequent hypercyclicity and frequent hypercyclicity:
 \[UFHC(T)=\bigcup_{a\in \mathcal{A}}FHC_a(T) \quad \text{and}\quad FHC(T)=\bigcap_{a\in \mathcal{A}}UFHC_a(T).\]
In particular, every operator $\mathcal{U}$-frequently hypercyclic is $FHC_a$ for some $a\in \mathcal{A}$. It is natural to wonder if an operator $UFHC_a$ for every $a\in \mathcal{A}$ is also frequently hypercyclic. This implication would allow us to use Baire arguments in the study of frequent hypercyclicity in view of results obtained by Bonilla and Grosse-Erdmann in~\cite{Bon}. Unfortunately, we will show in Section~\ref{Sec2} that there exist operators on $\ell_1(\mathbb{N})$ which are $UFHC_a$ for every $a\in \mathcal{A}$ but not frequently hypercyclic. However, we do not know if such an operator can also be found in Hilbert spaces.

\begin{problem}
Does there exist an operator on $\ell_p(\mathbb{N})$ for $1<p<\infty$ which is $UFHC_a$ for every $a\in \mathcal{A}$ but not frequently hypercyclic?
\end{problem}

%\begin{prop}
%For every set $I$, there exists $a\in A$ such that $\overline{d_a}(I)=\overline{d}(I)$
%\end{prop}
%\begin{proof}
%Let $I=(n_k)$ be a set such that $\overline{d}(I)=\delta$. This means that $\limsup \frac{k}{n_k}=\delta$. In other words, there exists an increasing sequence $(k_l)$ such that
%\[\frac{k_{l+1}-k_l}{n_{k_{l+1}}-n_{k_l}+\sum_{s=0}^{l-1} (n_{k_{s+1}}-n_{k_{s}})\frac{l+1}{s+1}}\ge (1-\frac{1}{l})\delta.\]
%Therefore, if we consider $a=(a_n)_{n\ge 0}$ given by $a_n=\frac{1}{l+1}$ if $n\in (n_{k_l},n_{k_{l+1}}]$ then
%\begin{align*}
%\frac{\sum_{j\in[0,n_{k_{l+1}}]\cap I}a_j}{\sum_{j\in[0,n_{k_{l+1}}]\cap I}}&\ge \frac{\frac{|(n_{k_{l}},n_{k_{l+1}}]\cap I|}{l+1}}{\sum_{s=0}^{l} \frac{n_{k_{s+1}}-n_{k_{s}}}{s+1}}=\frac{\frac{k_{l+1}-k_l}{l+1}}{\sum_{s=0}^{l} \frac{n_{k_{s+1}}-n_{k_{s}}}{s+1}}\\
%&=\frac{k_{l+1}-k_l}{n_{k_{l+1}}-n_{k_l}+\sum_{s=0}^{l-1} (n_{k_{s+1}}-n_{k_{s}})\frac{l+1}{s+1}}\ge (1-\frac{1}{l})\delta
%\end{align*}
%and thus $\overline{d_a}(I)=\delta=\overline{d}(I)$.
%\end{proof}
%
%Since for every $a\in \mathcal{A}$ we have $\overline{d_a}(I)\le\overline{d}(I)$, we get the following consequence.
%
%\begin{cor}
%For every set $I$, we have 
%\[\sup_{a\in A}\overline{d_a}(I)= \overline{d}(I).\]
%\end{cor}

\section{Weighted densities between $\underline{d}$ and $\overline{d}$}\label{Sec1}

We start by showing that for every set $I$ of non-negative integers, there exists a weight $a\in \mathcal{A}$ such that the upper density $\overline{d}(I)$ coincides with the weighted lower density $\underline{d}_a(I)$. This result will imply that 
\[\sup_{a\in \mathcal{A}}\underline{d}_{a}=\overline{d}\quad\text{ and }\quad \inf_{a\in \mathcal{A}}\overline{d}_{a}=\underline{d}.\]

\begin{theorem}\label{thm1}
For every set $I$ of non-negative integers, there exists $a\in \mathcal{A}$ such that
\[\underline{d_a}(I)= \overline{d_a}(I)=\overline{d}(I).\]
\end{theorem}
\begin{proof}
Let $I$ be a set of non-negative integers such that $\overline{d}(I)=\delta$. We consider an increasing sequence $(n_k)$ with $n_0=0$ tending to infinity such that for every $k\ge 0$, $\frac{|[n_k,n_{k+1})\cap I|}{n_{k+1}-n_k}\ge (1-\frac{1}{2^k})\delta$ and such that $n_{k+1}-n_k>n_k-n_{k-1}$ for every $k\ge 1$. We then let $a_n=\frac{1}{n_{k+1}-n_{k}}$ if $n\in [n_{k},n_{k+1})$ and we remark that $a\in \mathcal{A}$. Moreover, if $n_{k}< N\le n_{k+1}$ with $k\ge 1$, we have
\begin{align*}
\frac{\sum_{j\in [0,N)\cap I}a_{j}}{\sum_{j\in [0,N)}a_j}\ge 
\frac{\sum_{j\in [0,n_{k})\cap I}a_{j}}{\sum_{j\in [0,n_{k+1})}a_j}
\ge \frac{\sum_{j=1}^{k-1}(1-\frac{1}{2^j})\delta}{k+1}=\delta\Big(\frac{k-1}{k+1}-\frac{\sum_{j=1}^{k-1}\frac{1}{2^j}}{k+1}\Big)\to \delta.
\end{align*}
We deduce that $\underline{d_a}(I)=\delta$ and thus that \[\underline{d_a}(I)= \overline{d_a}(I)=\overline{d}(I).\]
\end{proof}
\begin{cor}
For every set $I$ of non-negative integers, we have 
\[\sup_{a\in \mathcal{A}}\underline{d_a}(I)= \sup_{a\in \mathcal{A}}\overline{d_a}(I)= \overline{d}(I).\]
\end{cor}

\begin{cor}
For every set $I$ of non-negative integers, there exists $a\in \mathcal{A}$ such that
\[\underline{d_a}(I)= \overline{d_a}(I)=\underline{d}(I)\]
and thus 
\[\inf_{a\in \mathcal{A}}\underline{d_a}(I)= \inf_{a\in \mathcal{A}}\overline{d_a}(I)= \underline{d}(I).\]
\end{cor}
\begin{proof}
Let $I$ be a set of non-negative integers and $a\in \mathcal{A}$. Since we have $\underline{d_a}(\mathbb{N}\backslash I)=1-\overline{d_a}(I)$, the result follows from Theorem~\ref{thm1}.
\end{proof}

This last corollary allows to express the set of frequently hypercyclic vectors in terms of $\mathcal{U}$-frequently hypercyclic vectors with respect to weights in $\mathcal{A}$.

\begin{theorem}\label{FHC1}
Let $X$ be a separable infinite-dimensional Fréchet space and $T$ a continuous linear operator on $X$. Then
\[FHC(T)=\bigcap_{a\in \mathcal{A}}FHC_a(T)=\bigcap_{a\in \mathcal{A}}UFHC_a(T).\]
\end{theorem} 
\begin{proof}
We know that 
\[FHC(T)\subset\bigcap_{a\in \mathcal{A}}FHC_a(T)\subset\bigcap_{a\in \mathcal{A}}UFHC_a(T).\]
On the other hand, if $x\in \bigcap_{a\in \mathcal{A}}UFHC_a(T)$ and $(U_n)$ is an open basis of $X$ then we have $\overline{d_a}(N(x,U_n))>0$ for every $n$ and every $a\in \mathcal{A}$. Therefore, since for every $n$ there exists $a\in \mathcal{A}$ such that $\underline{d}(N(x,U_n))=\overline{d_a}(N(x,U_n))$, we deduce that $\underline{d}(N(x,U_n))>0$ for every $n$.
\end{proof}

We can obtain a similar result for the $\mathcal{U}$-frequent hypercyclicity. However, we first need to adapt Theorem~\ref{thm1}.

\begin{theorem}\label{thm seq}
Let $(I_n)_{n\ge 1}$ be a sequence of sets of positive upper density. There exists $a\in \mathcal{A}$ such that $\underline{d_a}(I_n)>0$ for every $n\ge 1$.
\end{theorem}
\begin{proof}
Let $(I_n)_{n\ge 1}$ be a sequence of sets of positive upper density and $\delta_n= \overline{d}(I_n)$. We consider $(A_n)_{n\ge 1}$ a partition of $\mathbb{N}$ such that each set $A_n$ is infinite and has bounded gaps. We let $n_0=0$ and we select an increasing sequence $(n_l)_{l\ge 1}$ such that $(n_{l+1}-n_l)_{l\ge 1}$ is an increasing sequence tending to infinity and such that for every $n \ge 1$, every $l \in A_n$,
\[\frac{|[n_{l},n_{l+1})\cap I_n|}{n_{l+1}-n_{l}}\ge \left(1-\frac{1}{2^l}\right)\delta_n.\] We then let $a_n=\frac{1}{n_{l+1}-n_{l}}$ if $n\in [n_{l},n_{l+1})$ and we remark that $a\in \mathcal{A}$. Moreover, for every $n\ge 1$, if $(l_j)_{j\ge 1}$ is the increasing enumeration of $A_n$ and if $n_{l_j}< N\le n_{l_{j+1}}$, we have 
\begin{align*}
\frac{\sum_{k\in [0,N)\cap I_{n}}a_{k}}{\sum_{k\in [0,N)}a_k}&\ge 
\frac{\sum_{k\in [0,n_{l_j})\cap I_{n}}a_{k}}{\sum_{k\in [0,n_{l_{j+1}})}a_k}\\
&\ge \frac{\sum_{j'=1}^{j-1}\sum_{k\in [n_{l_{j'}},n_{l_{j'}+1})\cap I_{n}}a_{k}}{l_{j+1}}\\
&\ge \frac{\sum_{j'=1}^{j-1}\left(1-\frac{1}{2^{l_{j'}}}\right)\delta_n}{l_{j+1}}\\
&\ge \frac{(j-3)\delta_n}{l_{j+1}}.
\end{align*}
Finally, since $A_n$ is a set with bounded gaps, there exists $R_n$ such that for every $j\ge 1$, $l_j\le jR_n$ and we conclude that 
\[\underline{d_a}(I_n)=\liminf_N \frac{\sum_{k\in [0,N)\cap I_{n}}a_{k}}{\sum_{k\in [0,N)}a_k}\ge \frac{\delta_n}{R_n}>0.\]
\end{proof}

\begin{cor}
Let $X$ be a separable infinite-dimensional Fréchet space and $T$ a continuous linear operator on $X$. Then
\[UFHC(T)=\bigcup_{a\in \mathcal{A}}UFHC_a(T)=\bigcup_{a\in \mathcal{A}}FHC_a(T).\] 
\end{cor}
\begin{proof}
Let $(U_n)_{n\ge 1}$ be an open basis of $X$. If $T$ is not $\mathcal{U}$-frequently hypercyclic then the result is obvious since $FHC_a(T)\subset UFHC_a(T)\subset UFHC(T)$.
On the other hand, if $T$ is $\mathcal{U}$-frequently hypercyclic and $x$ is a $\mathcal{U}$-frequently hypercyclic vector for $T$, then we have $\overline{d}(N(x,U_n))>0$ for every $n\ge 1$ and it follows from Theorem~\ref{thm seq} that there exists $a\in \mathcal{A}$ such that $\underline{d_a}(N(x,U_n))> 0$ for every $n\ge 1$. Since $(U_n)$ is an open basis, this implies that $x$ is frequently hyperyclic with respect to $a$ and thus
\[UFHC(T)\subset \bigcup_{a\in \mathcal{A}}FHC_a(T)\subset \bigcup_{a\in \mathcal{A}}UFHC_a(T)\subset UFHC(T).\]
\end{proof}

Since we have already mentionned, the set $UFHC(T)$ of $\mathcal{U}$-frequently hypercyclic vectors is either empty or residual~\cite{BayartR, GrivM, Moot} and the set $FHC(T)$ of frequently hypercyclic vectors is always meager~\cite{BayartR}. These results can be generalized to weighted versions of $\mathcal{U}$-frequent hypercyclicity and of frequent hypercyclicity.

\begin{theorem}
Let $X$ be a separable infinite-dimensional Fréchet space and $T$ a continuous linear operator on $X$.
For every $a\in \mathcal{A}$, $UFHC_a(T)$ is either empty or residual and $FHC_a(T)$ is a meager set.
\end{theorem}
\begin{proof}
The result for $UFHC_a(T)$ follows from the fact that the family of sets with positive density $\overline{d_{a}}$ is an upper Furstenberg family \cite[Example 12(d) and Theorem~15]{Bon}.
On the other hand, for the set $FHC_a(T)$, if we look at the proof given by Moothathu~\cite[Theorem 1]{Moot} in the case of frequent hypercyclicity, it is sufficient to show that $\underline{d}_a(I)=\underline{d}_a(I+1)$ for every set $I$ of non-negative integers. 

We first remark that $\underline{d}_a(I)\ge \underline{d}_a(I+1)$. Indeed, since $a$ is decreasing, we have for every $N$,
\[\frac{\sum_{k\in [0,N]\cap I} a_k}{\sum_{k\in [0,N]}a_k}\ge \frac{\sum_{k\in [0,N]\cap (I+1)} a_k}{\sum_{k\in [0,N]}a_k}.\]
On the other hand, if we denote by $(n_k)_{k\ge 1}$ the increasing enumeration of indices such that $\frac{a_{n_k+1}}{a_{n_k}}\le \alpha$ where $\alpha$ is a fixed positive real number strictly smaller than~$1$. We can then deduce that $a_{n_k}\le a_{0}\alpha^{k-1}$ since $a$ is decreasing and thus
\[\frac{\sum_{j\in [0,N]\cap (n_k)}a_{j}}{\sum_{j=0}^{N}a_j}\le \frac{\sum_{k: n_k\le N}a_{0}\alpha^{k-1}}{\sum_{j=0}^{N}a_j}\xrightarrow[N\to \infty]{} 0.\]
%\[\frac{\sum_{l=1}^ka_{n_l}}{\sum_{j=0}^{n_k}a_j}\le \frac{\sum_{l=1}^ka_{0}\alpha^{l-1}}{\sum_{j=0}^{n_k}a_j}\to 0.\]
Therefore, for every $N\ge 0$, we have 
\begin{align*}
\frac{\sum_{j\in [0,N]\cap (I+1)}a_j}{\sum_{j=0}^{N}a_j}&= \frac{\sum_{j\in [0,N)\cap I}a_{j+1}}{\sum_{j=0}^{N}a_j}\\
&\ge \frac{\sum_{j\in ([0,N)\cap I)\backslash (n_k)}\frac{a_{j+1}}{a_j} a_j}{\sum_{j=0}^{N}a_j}\\
&\ge  \frac{\sum_{j\in ([0,N)\cap I)\backslash (n_k)}\alpha a_j}{\sum_{j=0}^{N}a_j}\\
&\ge\alpha\frac{\sum_{j\in [0,N)\cap I} a_j}{\sum_{j=0}^{N}a_j}-\alpha\frac{\sum_{j\in [0,N]\cap (n_k)}a_{j}}{\sum_{j=0}^{N}a_j}.
\end{align*}
and since $\frac{\sum_{j\in [0,N]\cap (n_k)}a_{j}}{\sum_{j=0}^{N}a_j}\to 0$, we get $\underline{d}_a(I+1)\ge \alpha\underline{d}_a(I)$. Finally, this inequality is satisfied for every $\alpha<1$ and thus $\underline{d}_a(I)=\underline{d}_a(I+1)$.
\end{proof}
\begin{remark}
It follows from this theorem that it is not possible to find a countable set $\mathcal{B}\subset \mathcal{A}$ such that $FHC(T)=\bigcap_{b\in\mathcal{B}}UFHC_b(T)$.
\end{remark}

\section{Difference between UFHC$_a$ and FHC}\label{Sec2}

In view of Theorem~\ref{FHC1}, we can wonder if an operator $UFHC_a$ for every $a\in \mathcal{A}$ is also frequently hypercyclic. Unfortunately, this is not the case. Indeed, we will show in this section how we can construct an operator which is UFHC$_a$ for every $a \in \mathcal{A}$ and not frequently hypercyclic. However, we will only be able to exhibit such an operator on $\ell_1(\N)$, while examples of operators which are $\mathcal{U}$-frequently hypercyclic and not frequently hypercyclic were given by Bayart and Ruzsa on $c_0(\N)$~\cite{BayartR} and by Grivaux, Matheron and Menet on $\ell_p(\mathbb{N})$~\cite{Monster} for every $1\le p<\infty$. 

In order to identify operators which are $UFHC_a$ for every $a\in \mathcal{A}$, we first need a criterion for $\mathcal{U}$-frequent hypercyclicity with respect to a weight $a$. To this end, we adapt the condition based on periodic points for $\mathcal{U}$-frequent hypercyclicity given in \cite[Theorem 5.14]{Monster}.
	
\begin{theorem}\label{UFHCa}
Let $a\in \mathcal{A}$ and $X_0$ a dense subspace of $X$ with $T(X_0)\subset X_0$ and $X_0\subset \text{Per}(T)$. There exists a non-decreasing sequence of positive integers $(\alpha_n)_{n\ge 1}$ such that if for every $x\in X_0$ and every $\varepsilon>0$, there exist $z\in X_0$ and $n\ge 1$ such that 
\begin{enumerate}
\item $\|z\|<\varepsilon$
\item $\|T^{n+k}z-T^kx\|<\varepsilon$ for every $0\le k\le \alpha_n n$.
\end{enumerate}
then $T$ is $UFHC_a$.
\end{theorem}
\begin{proof}
Let $a\in \mathcal{A}$ and $X_0$ a dense subspace of $X$ with $T(X_0)\subset X_0$ and $X_0\subset \text{Per}(T)$. There exists a non-decreasing sequence of positive integers $(\alpha_n)_{n\ge 1}$ such that for every $n\ge 1$, 
\begin{equation}
\frac{\sum_{k=n}^{(1+\alpha_{n})n-1} a_k}{\sum_{k=0}^{(1+\alpha_{n})n}a_k}\ge \frac{1}{2}.
\label{eq}
\end{equation}
Let $(x_l)_{l\ge 1}$ be a dense sequence in $X_0$, $(I_l)_{l\ge 1}$ a partition of $\mathbb{N}$ such that for every $l$, $I_l$ is an infinite set and $y_j=x_l$ if $j\in I_l$. If for every $x\in X_0$ and every $\varepsilon>0$, there exist $z\in X_0$ and $n\ge 1$ such that 
\begin{enumerate}
\item $\|z\|<\varepsilon$
\item $\|T^{n+k}z-T^kx\|<\varepsilon$ for every $0\le k\le \alpha_n n$,
\end{enumerate}
then we can construct a sequence $(z_j)_{j\ge 0}$ in $X_0$ and a strictly increasing sequence $(n_j)_{j\ge 0}$ with $z_0=0$ and $n_0=0$ such that for every $j\ge 1$,
\begin{enumerate}
\item $\|T^k z_j\|<2^{-j}$ for every $0\le k\le (1+\alpha_{n_{j-1}})n_{j-1}$;
\item $\|T^{n_j+k}z_j-T^k(y_j-\sum_{i<j}z_i)\|<2^{-j}$ for every $0\le k\le \alpha_{n_j} n_j$;
\item $n_j$ is a multiple of the period of the vector $\sum_{i<j}z_i$.
\end{enumerate}
We can assume that $n_j$ is a multiple of the period of the vector $\sum_{i<j}z_i$ because if $d$ is the period of $\sum_{i<j}z_i$, if $\|T^k z\|<2^{-j}\|T\|^{-d}$ for every $0\le k\le (1+\alpha_{n_{j-1}})n_{j-1}$ and if $\|T^{n+k}z-T^k(y_j-\sum_{i<j}z_i)\|<2^{-j}$ for every $0\le k\le \alpha_{n} n$, then there exists $d'\le d$ such that $n_j=n-d'$ is a multiple of $d$ and such that $z_j=T^{d'}z$ satisfies $(1)$ and $(2)$ since the sequence $(\alpha_k k)_k$ is increasing.

By letting $z=\sum_{i\ge 1}z_i$, we can then show that for every $l\ge 1$, every $\varepsilon>0$, there exists $j_0$ such that for every $j\ge j_0$, $j\in I_l$, the set $N(z,B(x_l,\varepsilon))$ contains $\{n_j+m\text{per}(x_l):0\le m\le \alpha_{n_j}n_j/\text{per}(x_l)\}$ (see \cite[Proof of Theorem 5.14]{Monster}).

Therefore, since the sequence $(a_k)_k$ is decreasing, we get for every $l\ge 1$, every $\varepsilon>0$,
\begin{align*}
\overline{d}_a(N(z,B(x_l,\varepsilon)))&\ge \limsup_{j\in I_l}
\frac{\sum_{m=0}^{\alpha_{n_j}n_j/\text{per}(x_l)} a_{n_j+m\text{per}(x_l)}}{\sum_{k=0}^{(1+\alpha_{n_j})n_j}a_k}\\
&\ge \limsup_{j\in I_l}
\frac{\sum_{k=n_j}^{(1+\alpha_{n_j})n_j-1} \frac{a_{k}}{\text{per}(x_l)}}{\sum_{k=0}^{(1+\alpha_{n_j})n_j}a_k}\ge \frac{1}{2\text{per}(x_l)} \quad\text{by \eqref{eq}.}
\end{align*}
Since $(x_l)$ is a dense sequence, we deduce that $z\in UFHC_a(T)$ and thus that $T$ is $UFHC_a$.
\end{proof}

In order to exhibit an operator which is UFHC$_a$ for every $a\in \mathcal{A}$ and not frequently hypercyclic, we will consider operators of C-type which have been introduced in \cite{Menet} to get a chaotic operator which is not frequently hypercyclic and have allowed to exhibit $\mathcal{U}$-frequently hypercyclic operators on Hilbert spaces which are not frequently hypercyclic \cite{Monster}.

An operator of C-type is associated to four 
parameters $v$, $w$, $\varphi $, and $b$ where

\begin{enumerate}
 \item[-] $v=(v_{n})_{n\ge 1}$ is a bounded sequence of non-zero complex numbers;
\item[-] $w=(w_{j})_{j\geq 1}$ is a sequence of complex numbers which is both bounded  
and bounded below, \mbox{\it i.e.} $0<\inf_{k\ge 1} \vert w_k\vert\leq \sup_{k\ge 1}\vert w_k\vert<\infty$;
\item[-] $\varphi $ is a map from $\N$ into itself, such that $\varphi 
(0)=0$, $\varphi (n)<n$ for every $n\ge 1$, and the set 
$\varphi ^{-1}(l)=\{n\ge 0\,:\,\varphi (n)=l\}$ is infinite for every 
$l\ge 0$;
\item[-] $b=(b_{n})_{n\ge 0}$ is a strictly increasing sequence of positive 
integers such that $b_{0}=0$ and $b_{n+1}-b_{n}$ is a multiple of 
$2(b_{\varphi (n)+1}-b_{\varphi (n)})$ for every $n\ge 1$.
\end{enumerate}

\begin{definition}\label{Definition 43}
 The \emph{operator of C-type} $T_{v,w,\varphi,b}$ associated to the data  
$v$, $w$, $\varphi $, and $b$ given as above is defined by
\[
T_{v,w,\varphi,b}\ e_k=
\begin{cases}
 w_{k+1}\, e_{k+1} & \textrm{if}\ k\in [b_{n},b_{n+1}-1),\; n\geq 0,\\
v_{n}\, e_{b_{\varphi(n)}}-\Bigl(\,\,\prod_{j=b_{n}+1}^{b_{n+1}-1}
w_{j}\Bigr)^{ -1 } e_{
b_{n}} & \textrm{if}\ k=b_{n+1}-1,\ n\ge 1,\\
 -\Bigl(\!\!\prod_{j=b_0+1}^{b_{1}-1}w_j\Bigr)^{-1}e_0& \textrm{if}\ 
k=b_1-1.
\end{cases}
\]

\end{definition}

We remark that the operator $T_{v,w,\varphi,b}$ is  well-defined and bounded on $\ell_1(\N)$ as soon as $\inf_{n\geq 0} \prod_{j=b_n+1}^{b_{n+1}-1} \vert w_j\vert >0$. From now on, we will always assume that this condition is satisfied. In \cite{Monster}, it was also assumed that $\sum_n |v_n|<\infty$ in order to get a continuous operator on $\ell_p(\N)$ for any $1\le p<\infty$. However, in our construction, we will need to consider a sequence $(v_n)$ which takes infinitely often the same values and we will thus restrict ourselves to operators on $\ell_1(\N)$ so that such a sequence $(v_n)$ can be considered.

Moreover, we impose the following restrictions on the parameters 
$v$, $w$, $\varphi$ and $b$. We first consider a map $\psi(k)=(\psi_1(k),\psi_2(k))$ such that for every $k\ge 1$, 
\begin{equation}
\label{psi2}
1\le \psi_1(k)< \min\{k+1,\psi_2(k)\}\quad\text{and}\quad \psi_2(k)> \max\{\psi_2(j):j< \psi_1(k)\}
\end{equation} and such that for every $i\ge 1$, there exists $j_i$ such that for every $j\ge j_i$, the set $\{k\ge 1:\psi(k)=(i,j)\}$ is infinite. We let $n_0=0$, $n_1=1$ and $n_{k+1}=n_k+\psi_1(k)$ for every $k\ge 1$. For every $k\ge 1$, every $n\in [n_k,n_{k+1})$, we then let $\varphi(n)=n-n_k$, $
v_n=v^{(k)}=2^{-\tau^{(\psi_2(k))}}$, $b_{n+1}-b_n=\Delta^{(k)}$ and for every $j\in [b_n+1,b_{n+1}-1]$,
\[
w_j=
\begin{cases}
  2 & \quad\text{if}\ \ b_n< j\le b_n+k\delta^{(\psi_2(k))}+2n+1\\
 1 & \quad\text{if}\ \ b_n+k\delta^{(\psi_2(k))}+2n+1< j<\Delta^{(k)}-3\delta^{(\psi_2(k))}-2n-1\\
 1/2 & \quad\text{if}\ \  \Delta^{(k)}-3\delta^{(\psi_2(k))}-2n-1\le i 
<\Delta^{(k)}-2\delta^{(\psi_2(k))}\\
 2 & \quad\text{if}\ \ \Delta^{(k)}-2\delta^{(\psi_2(k))}\le i< \Delta^{(k)}-\delta^{(\psi_2(k))}\\
 1& \quad\text{if}\ \ \Delta^{(k)}-\delta^{(\psi_2(k))}\le i<\Delta^{(k)}
\end{cases}
\]
where $(k+3)\delta^{(\psi_2(k))}+4n_{k+1}+2<\Delta^{(k)}$ for every $k\ge 1$, $(\delta^{(k)})$ and $(\tau^{(k)})$ are increasing sequences of positive integers and $\Delta^{(k)}$ is a multiple of $2\Delta^{(k-1)}$ with $\Delta^{(0)}=b_1-b_0$. Observe that for every $n\in [n_k,n_{k+1})$, we have
\[
\prod_{j=b_n+1}^{b_{n+1}-1}w_j=\prod_{i=1}^{\Delta^{(k)}-1}w_{b_n+i}=2^{k\delta^{(\psi_2(k))}}.
\]

The operator $T_{v,w,\varphi,b}$ associated to these parameters is then an operator of C-type on $\ell_1(\N)$ which will be denoted by $T$ from now and we will show that for a convenient choice of parameters $(\delta^{(k)})$, $(\tau^{(k)})$ and $(\Delta^{(k)})$, $T$ is an operator which is UFHC$_a$ for every $a\in \mathcal{A}$ and not frequently hypercyclic.\\

We first show that $T$ is UFHC$_{a}$ for every $a\in \mathcal{A}$ if $\delta^{(k)}-\tau^{(k)}$ tends to infinity by applying Theorem~\ref{UFHCa} with $X_0=\textrm{span}\,[e_{k}\,:\,k\ge 0]$ since operators of C-type have the property that for every $n\ge 0$, every $j\in[b_n,b_{n+1})$, $T^{2(b_{n+1}-b_n)}e_j=e_j$ (\cite[Lemma 6.4]{Monster}).

\begin{theorem}\label{UFHC}
If $\delta^{(k)}-\tau^{(k)}$ tends to infinity then $T$ is UFHC$_{a}$ for every $a\in \mathcal{A}$.
\end{theorem}

\begin{proof}
Let $a\in \mathcal{A}$ and $X_{0}=\textrm{span}\,[e_{k}\,:\,k\ge 0]$. We consider the non-decreasing sequence of positive integers $(\alpha_n)_{n\ge 1}$
 given by Theorem \ref{UFHCa} and we show that we can apply this one to $T$ if $\delta^{(k)}-\tau^{(k)}$ tends to infinity. Let $x\in X_{0}$ and
 $\varepsilon >0$. There exists $k_0\geq 1$ such that $x$ may be written as 
\[x=\sum_{l<n_{k_0}}\sum_{j=b_l}^{b_{l+1}-1}
x_{j}e_{j}.\]

Let $j_{n_{k_0}}$ such that for every $j\ge j_{n_{k_0}}$, the set $\{k\ge 1:\psi(k)=(n_{k_0},j)\}$ is infinite. We choose $K\ge j_{n_{k_0}}$ such that $\delta^{(K)}>b_{n_{k_0}}$ and such that
\[\frac{\|x\| (\sup_i |w_i|)^{b_{n_{k_0}}}}{2^{\delta^{(K)}-\tau^{(K)}}(\inf_i |w_i|)^{b_{n_{k_0}}}}<\varepsilon\] and we then choose $k\ge 1$ such that $k\ge 2\alpha_{2\delta^{(K)}}+1$, $\psi_1(k)= n_{k_0}$ and $\psi_2(k)=K$.
 It follows that for every $n\in [n_k,n_{k+1})$,
\begin{align}
|v^{(k)}|\prod_{i=b_{n+1}-2\delta^{(K)}}^{b_{n+1}-1}|w_{i}|=2^{\delta^{(K)}-\tau^{(K)}}
\label{Equation 9}
\end{align}
and
\begin{align}
|v^{(k)}|\prod_{i=b_n+m+1}^{b_{n+1}-1}|w_i|\ge 2^{\delta^{(K)}-\tau^{(K)}} \qquad
\textrm{for every}\ 0\le m\le \alpha_{2\delta^{(K)}} 2\delta^{(K)},
\label{Equation 10}
\end{align}
since $\alpha_{2\delta^{(K)}} 2\delta^{(K)}\le (k-1)\delta^{(K)} $ by assumption on $k$.
We then set
\begin{align*}
z:=\sum_{l<n_{k_0}}\ 
\sum_{j=b_l}^{b_{l+1}-1}
\Bigg[&x_{j}\
\Big({v^{(k)}}\!\!\!\!\!\!\!\!\!\;\;\;\;\;\;\prod_{
i=b_{n_k+l+1}-2\delta^{(K)}+j-b_{l}+1}^{b_{n_k+l+1}-1}w_{i}\Big)^{-1}\\
&\quad\Big(\,\prod_{i=1}^{j-b_l}w_{b_l+i} \Big)^{-1}
\ e_{b_{n_{k}+l+1}-2\delta^{(K)}+j-b_{l}}\Bigg].
\end{align*}
\par\smallskip
It remains to prove that \[ \|z\|<\varepsilon\qquad {\rm and}\qquad \|T^{\,2\delta^{(K)}+m}z-T^{\,m}x\|
<\varepsilon\quad\hbox{for every $0\le m\le \alpha_{2\delta^{(K)}} 2\delta^{(K)}$}
\] 
since the desired result will then follow from Theorem \ref{UFHCa} with $n=2\delta^{(K)}$. Note that the key point in the construction of $z$ relies on the fact that the weights appearing in the definition of $z$ and in particular the choice of $n$ only depend on $K$ while the orbit of $T^nz$ will follow the orbit of $x$ during a time depending on the size of the first block of 2 in $(w_j)_{j\in [b_n,b_{n+1})}$ for $n\in [n_k,n_{k+1})$ which can be arbitrarily big if $k$ is sufficiently big.
\par\smallskip 

We remark by applying \eqref{Equation 9} that
\begin{align*}
\|z\|&\le \|x\| \sup_{l<n_{k_0}}\sup_{j\in [b_l,b_{l+1})}\left({|v^{(k)}|}\!\!\!\!\!\!\!\!\!\;\;\;\;\;\;\prod_{
i=b_{n_k+l+1}-2\delta^{(K)}+j-b_{l}+1}^{b_{n_k+l+1}-1}|w_{i}|\right)^{-1}
\Bigl(\,\prod_{i=1}^{j-b_l}|w_{b_l+i}| \Bigr)^{-1}\\
&\le \frac{\|x\| (\sup_i |w_i|)^{b_{n_{k_0}}}}{2^{\delta^{(K)}-\tau^{(K)}}(\inf_i |w_i|)^{b_{n_{k_0}}}}< \varepsilon.
\end{align*}

Let us now estimate the norm of the vector $T^{\,2\delta^{(K)}+m}z-T^{\,m}x$ for every $0\le 
m\le \alpha_{2\delta^{(K)}} 2\delta^{(K)}$. 
Note that if $0\le l< n_{k_0}$ and 
$b_{l}\le j<b_{l+1}$, then 
$2\delta^{(K)}-(j-b_{l})> 1$ since $0\le j-b_{l}<b_{n_{k_0}}$ and $\delta^{(K)}>b_{n_{k_0}}$. Let $0\le l <n_{k_0}$. Since $n_k+l<n_k+\psi_1(k)=n_{k+1}$, we have for every $b_l\le j<b_{l+1}$,
\begin{align*}
T^{\,2\delta^{(K)}-(j-b_l)}e_{b_{n_k+l+1}-2\delta^{(K)}+j-b_{l}}=
&\Bigl(v^{(k)}\prod_{i=b_{n_k+l+1}-2\delta^{(K)}+j-b_{l}+1}^{b_{n_k+l+1}-1} w_i\Bigr)\, e_{b_l}\\
&-\Biggl(\; \prod_{i=b_{n_k+l}+1}^{b_{n_k+l+1}-2\delta^{(K)}+j-b_{l}} w_i\Biggr)^{-1} e_{b_{n_k+l}},
\end{align*}
and thus
\begin{align*}
 T^{\,2\delta^{(K)}}e_{b_{n_k+l+1}-2\delta^{(K)}+j-b_{l}}=
\quad \Biggl( v^{(k)}&
\!\!\!\!\!\;\;\;\;\prod_{i=b_{n_k+l+1}-2\delta^{(K)}+j-b_{l}+1}^{b_{n_k+l+1}-1} w_i\Biggr)\ 
\Bigl(\,\,\prod_{i=1}^{j-b_l}w_{b_l+i} \Bigr) \,
e_{j} \\
&- \Biggl(\!\!\!\!\!\;\;\;\;\prod_{i=b_{n_k+l}+j-b_l+1}^{b_{n_k+l+1}-2\delta^{(K)}+j-b_{l}} w_i\Biggr)^{-1} 
\, e_{b_{n_k+l}+j-b_{l}}.
\end{align*}
Therefore,
\begin{align*}
T^{2\delta^{(K)}}z=x-\sum_{l<n_{k_0}}\ 
\sum_{j=b_l}^{b_{l+1}-1}\Bigg[
&x_{j}\
\Big({v^{(k)}}\!\!\!\!\!\!\!\!\!\;\;\;\;\;\;\prod_{
i=b_{n_k+l+1}-2\delta^{(K)}+j-b_{l}+1}^{b_{n_k+l+1}-1}w_{i}\Big)^{-1}
\Big(\,\prod_{i=1}^{j-b_l}w_{b_l+i} \Big)^{-1}\\
&\quad\Big(\!\!\!\!\!\;\;\;\;\prod_{i=b_{n_k+l}+j-b_l+1}^{b_{n_k+l+1}-2\delta^{(K)}+j-b_{l}} w_i\Big)^{-1} 
\, e_{b_{n_k+l}+j-b_{l}}\Bigg].
\end{align*}
\par\smallskip
Moreover, if $0\le m\le \alpha_{2\delta^{(K)}} 2\delta^{(K)}$ then for every $0\le l< n_{k_0}$ and 
every $b_{l}\le j<b_{l+1}$,
\[
T^{\,m}\,e_{b_{n_k+l}+j-b_{l}}=\Biggl(\;\,\prod_{i=b_{n_k+l}+j-b_l+1}^{b_{n_k+l}+j-b_{l}+m}
w_{i}\Biggr)\,e_{b_{n_k+l}+j-b_{l}+m},
\]
because $ j-b_{l}+\alpha_{2\delta^{(K)}} 2\delta^{(K)}<b_{l+1}-b_{l}+(k-1)\delta^{(K)}
<b_{n_{k_{0}}}+(k-1)\delta^{(K)}<k\delta^{(K)}
<\Delta^{(k)}-3\delta^{(K)}$. So we get 
%\begin{align*}
% T^{\,2\delta^{(K)}+m}e_{b_{n_k+l+1}-2\delta^{(K)}+j-b_{l}}=
%& \Biggl( v^{(k)}
%\!\!\!\!\!\;\;\;\;\prod_{i=b_{n_k+l+1}-2\delta^{(K)}+j-b_{l}+1}^{b_{n_k+l+1}-1}w_{i }\Biggr)\ 
%\Bigl(\,\,\prod_{i=1}^{j-b_l}w_{b_l+i} \Bigr) \,
%T^{m} e_{j} \\
%&- \Biggl(\!\!\!\!\!\;\;\;\;\prod_{i=b_{n_k+l}+j-b_l+1}^{b_{n_k+l+1}-2\delta^{(K)}+j-b_{l}}w_{i}\Biggr)^{-1} 
%\Biggl(\;\,\prod_{i=b_{n_k+l}+j-b_l+1}^{b_{n_k+l}+j-b_{l}+m}
%w_{i}^{(k)}\Biggr)
%\, e_{b_{n_k+l}+j-b_{l}+m}.
%\end{align*}
for any $0\leq m\leq \alpha_
{2\delta^{(K)}} 2\delta^{(K)}$,
\begin{align*}
T^{2\delta^{(K)}+m}z=T^{m}x-\sum_{l<n_{k_0}}\ 
\sum_{j=b_{l}}^{b_{l+1}-1}& \Bigg[x_{j}\Biggl(v^{(k)}\prod_{i=b_{n_k+l}+j-b_l+m+1}^{b_{n_k+l+1}-1}w_{i} 
\Biggr)^{-1}\\
&\Bigl(\,\,\prod_{i=1}^{j-b_l}w_{b_{l}+i}
\Bigr)^{-1}\,e_{b_{n_k+l}+j-b_{l}+m}\Bigg].
\end{align*}
We deduce from \eqref{Equation 10} that for any $0\leq m\leq \alpha_{2\delta^{(K)}} 2\delta^{(K)}$,

\[ \Vert T^{2\delta^{(K)}+m}z-T^{m}x\Vert< \frac{ \|x\| \sup_i |w_i|^{b_{n_{k_0}}}}{2^{\delta^{(K)}-\tau^{(K)}}\inf_i |w_i|^{b_{n_{k_0}}}}<\varepsilon.
\]
The assumptions 
of Theorem \ref{UFHCa} are thus satisfied and this concludes the 
proof.
\end{proof}

It remains to determine under which conditions $T$ is not frequently hypercyclic. To this end, it will be necessary to investigate the dynamical properties of finite sequences under the action of $T$. Therefore, we denote for every $l\ge 0$,
\[P_l x=\sum_{k=b_l}^{b_{l+1}-1}x_ke_k\quad \text{ and } X_{l}=\Bigl\|\sum_{k=b_{l}}^{b_{l+1}-1}\Bigl(\prod_{s=k+1}^{b_{l+1}-1}
w_{s}\Bigr)\,x_{k}e_{k}\Bigr\|.\]
We recall the following result proved in \cite{Monster}.
\begin{prop}[{\cite[Proposition 6.12]{Monster}}]\label{Proposition 50}
 Let $T$ be an operator of C-type on $\ell_{1}(\N)$, and  
 let $(C_{n})_{n\ge 
0}$ be a sequence of positive numbers with $0<C_n<1$. Assume that 
\[
|v_n|\ .\sup_{j\in 
[b_{\varphi(n)},b_{\varphi(n)+1})}\ \Bigl(
\prod_{s=b_{\varphi(n)}+1}^{j}|w_{s}|\Bigr)\le C_{n}\quad\hbox{  for every $n\ge 1$. }
\]
 Then, for any $x\in\ell_1(\N)$, 
we have for every $l\ge 1$ and 
every $0\le n<l$,
\begin{enumerate}
 \item [\rm {(1)}]$\quad
 {\sup_{j\ge 0}\ \|P_{n}T^{\,j}P_{l}\,x\|\le 
C_{l}\ X_l} $
\par\smallskip \noindent and
\item[\rm {(2)}]
$
\quad \sup_{j\le N}\|P_{n}T^{\,j}P_{l}\,x\|\le 
\displaystyle{C_l\Big(\sup_{b_{l+1}-N\le k<b_{l+1}}\prod_{s=k+1}^{b_{l+1}-1}|w_s|\Big) \|P_lx\|}
$
\par\smallskip \noindent for every $1\le N\le b_{l+1}-b_{l}$.
\end{enumerate}
\end{prop}

We will also need a simple adaptation of \cite[Proposition 6.13]{Monster}.

\begin{prop}\label{Proposition 51}
 Let $T$ be an operator of C-type on $\ell_{p}(\N)$
  and let $x\in
 \ell_p(\N)$.  
 Suppose that there exist two integers $0\le 
K_{0}<K_{1}\le b_{l+1}-b_{l}$ such that 
\[
|w_{b_{l}+k}|=1\quad\textrm{for every}\ 
k\in(K_{0},K_{1})\quad\textrm{and}\ 
\prod_{s=b_l+K_{0}+1}^{b_{l+1}-1}|w_{s}|=\alpha.
\]
Then we have for every $J\ge 0$,
\begin{align*}
 &\dfrac{1}{J+1}\ \#\Bigl\{0\le j\le J\,:\, \|P_{l}T^{\,j}P_{l}\,x\|\ge 
\alpha^{-1}X_{l}/2\Bigr\}\\
 &\hspace{3cm}\ge 
1-2\bigl(b_{l+1}-b_l-(K_1-K_0)\bigr)\,\cdot\,\Bigl( \frac{1}{J+1}+\frac{1}{b_{
l+1}-b_{l}}\Bigr)\cdot
\end{align*}
\end{prop}

We can now state sufficient conditions for $T$ not to be frequently hypercyclic by following the proof of \cite[Lemma 6.11]{Monster}.

\begin{theorem}\label{FHC}
For every $k\geq 1$, set $\gamma_k:=2^{k\delta^{(k-1)}+2n_{k}+1-\tau^{(k)}}$.
 Suppose that for every $k \ge 1$, $\gamma_k\le \min\{2^{-16k}, \min_{s\le k-1}\gamma^2_{s}\}$ and that the following two conditions are satisfied:
 \[ \lim_{k\to\infty }\dfrac{k\delta ^{(k)}}{\delta 
^{(k+1)}}=0 \quad \text{and} \quad \lim_{k\to\infty }\ \dfrac{k\delta ^{(\psi_2(k))}+n_{k+1}}{\Delta 
^{(k)}}=0.
\]
\par\smallskip \noindent
Then $T$ is not frequently hypercyclic.
\end{theorem}
\begin{proof}
Let $x$ be a hypercyclic vector of $T$. It suffices to show that $x$ cannot be a frequently hypercyclic vector.
We first consider $k_0\ge 1$ and $l_{0}\in [n_{k_0},n_{k_0+1})$ such that 
\[
\|P_{l_{0}}\,x\|\ge \frac{\gamma_{\psi_2(k_0)}^{1/2}}{2^{l_0}}\,\|x-P_{0}\,x\|.
\]
Such integers $k_0$ and $l_0$ exist because $\sum_{k\ge 1}\gamma_{\psi_2(k)}^{1/2}\sum_{l=n_{k}}^{n_{k+1}-1}\frac{1}{2^{l}}\le 1$ since $\gamma_k\le 1$ for every $k\ge 1$. Moreover, we remark that $x-P_{0}\,x$ is non-zero since $x$ is hypercyclic.

We now construct a strictly increasing  sequence of integers 
$(l_{m})_{m\ge 0}$ such that for every $m\ge 1$, if we set 
\[
 j_{m-1}:=\smash[b]{\min\,\biggl\{j\ge \,
0\,:\,\sum_{l>l_{m-1}}\|P_{l_{m-1}}T^{j}P_{l}\,x\|>X_{l_{m-1}}\biggr\},
}\]
and if $l_m\in [n_{k_m},n_{k_{m}+1})$ then
\begin{equation}
\psi_1(k_m)> l_{m-1},\quad
j_{m-1}>\delta^{(\psi_2(k_m))}\quad\textrm{and}\quad 
X_{l_{m}}\ge\frac{1}{2^{l_m}\gamma_{\psi_2(k_m)}^{1/2}}\,X_{l_{m-1}}.
\label{eq2}
\end{equation}

Assume that the integers $l_{1},\dots,l_{m-1}$ 
have already been constructed and satisfy~\eqref{eq2}. We first remark that the integer $j_{m-1}$ is well-defined since $x$ is hypercyclic and thus \[\limsup_j\sum_{l>l_{m-1}}\|P_{l_{m-1}}T^{j}P_{l}\,x\|=\infty.\]
Moreover, we can find an integer $S\ge 1$ 
such that 
\begin{equation}\label{Equation 16}
\sum_{k:\psi_2(k)=S}\sum_{l\in [n_k,n_{k+1})\cap(l_{m-1},\infty)} \|P_{l_{m-1}}T^{\,j_{m-1}}P_{l}\,x\|>\gamma
_S^{1/2}X_{l_{m-1}},
\end{equation}
because if we had $\sum_{k:\psi_2(k)=S}\sum_{l\in [n_k,n_{k+1})\cap(l_{m-1},\infty)} \|P_{l_{m-1}}T^{\,j_{n-1}}P_{l}\,x\|\le\gamma_S^{1/2}X_{l_{m-1}}$ for every $S\ge 1$, this would imply that 
$\sum_{l>l_{m-1}}\|P_{l_{m-1}}T^{\,j_{m-1}}P_{l}\,x\|\le 
\big
(\sum_{S=1}^{\infty}\gamma_S^{1/2}\big)X_{l_{m-1}}\le X_{l_{m-1}}$, violating the definition of $j_{m-1}$.
\par\smallskip
Therefore, there exist $k_m$ with $\psi_2(k_m)=S$ and $l_m\in [n_{k_m},n_{k_m+1})\cap(l_{m-1},\infty)$ such that
\begin{equation}
\|P_{l_{m-1}}T^{\,j_{m-1}}P_{l_m}\,x\|>\frac{\gamma_{\psi_2(k_m)}^{1/2}}{2^{l_m}}X_{l_{m-1}}.
\label{eq3}
\end{equation}
In particular, since $X_{l_{m-1}}>0$ and since $l_m>l_{m-1}$, we have $\psi_1(k_m)>l_{m-1}$.
Let $k\ge 1$, $l\in [n_k,n_{k+1})$ and $\kappa$ such that $\varphi(l)\in [n_{\kappa},n_{\kappa+1})$. Since $\kappa\le n_{\kappa}\le \varphi(l)<\psi_1(k)$, we have 
$\psi_{2}(\kappa)< \psi_{2}(k)$ by \eqref{psi2} and since $\varphi(l)<\psi_1(k)< \psi_{2}(k)\le n_{\psi_{2}(k)}$, we have
\begin{align*}
&|v_{l}|\ .\sup_{j\in 
[b_{\varphi(l)},b_{\varphi(l)+1})}\ \Bigl(
\prod_{s=b_{\varphi(l)}+1}^{j}|w_{s}|\Bigr)\\
&\quad \le 2^{-\tau^{(\psi_2(k))}}2^{\kappa\delta^{(\psi_2(\kappa))}+2\varphi(l)+1}\\
&\quad \le 2^{-\tau^{(\psi_2(k))}}2^{\psi_{2}(k)\delta^{(\psi_2(k)-1)}+2n_{\psi_{2}(k)}+1}=\gamma_{\psi_{2}(k)}.
\end{align*}
It follows from Proposition~\ref{Proposition 50} and \eqref{eq3} that
\[
X_{l_{m}}\ge\dfrac{1}{\gamma_{\psi_2(k_m)}}\|P_{l_{m-1}}T^{\,j_{m-1}}P_{l_{m}}\,x\|> \frac{1}{2^{l_m}\gamma_{\psi_2(k_m)}^{1/2}}X_{l_{m-1}}.
\]
On the other hand, for every $1\le s\le m$, we have $2^{l_{s-1}}\gamma_{\psi_2(k_s)}^{1/2}\le 2^{4l_{s-1}}\gamma_{\psi_2(k_s)}^{1/4}\le 2^{4(l_{s-1}-\psi_2(k_s))}\le 1$  since $\psi_2(k_s)>\psi_1(k_s)> l_{s-1}$. It follows from Proposition~\ref{Proposition 50} that for every $0\le 
j\le \delta^{(S)}=\delta^{(\psi_2(k_m))}$,
\begin{align*}
\sum_{k:\psi_2(k)=S}\sum_{l\in [n_{k},n_{k +1})\cap (l_{m-1},\infty)}\|P_{l_{m-1}}T^{\,j}P_{l}\,x\|&\le \gamma_{S}\sum_{k:\psi_2(k)=S}\sum_{l\in [n_{k},n_{k +1})\cap (l_{m-1},\infty)}\|P_{l}
\,x\|\\
&\le \gamma_{\psi_2(k_m)}\,\|x-P_{0}x\|\\
&\le \frac{2^{l_0}\gamma_{\psi_2(k_m)}}{\gamma_{\psi_2(k_0)}^{1/2}}\,\|P_{l_{0}}\,x\|\le \frac{2^{3l_0+1}\gamma_{\psi_2(k_m)}}{\gamma_{\psi_2(k_0)}^{1/2}}\,X_{l_0}\\
&\le \frac{2^{3l_0+1}\gamma_{\psi_2(k_m)}}{\gamma_{\psi_2(k_0)}^{1/2}}\Big(\prod_{s=1}^{m-1}2^{l_s}\gamma_{\psi_2(k_s)}^{1/2}\Big)X_{l_{m-1}}\\
&\le \frac{2^{l_{m-1}+2l_0+1}\gamma_{\psi_2(k_m)}}{\gamma_{\psi_2(k_0)}^{1/2}}X_{l_{m-1}}\\
&\le \frac{2^{4l_{m-1}}\gamma_{\psi_2(k_m)}}{\gamma_{\psi_2(k_0)}^{1/2}}X_{l_{m-1}}\le \gamma_{\psi_2(k_m)}^{1/2}X_{l_{m-1}}
\end{align*}
since $2^{4l_{m-1}}\gamma_{\psi_2(k_m)}^{1/4}\le 1$ and $\gamma_{\psi_2(k_m)}^{1/4}\le \gamma_{\psi_2(k_0)}^{1/2}$ because $\psi_1(k_m)> l_{m-1}\ge l_0\ge n_{k_0}\ge k_0$ and thus $\psi_2(k_0)< \psi_2(k_m)$ by \eqref{psi2}. We deduce from \eqref{Equation 16} that $j_{m-1}>\delta^{(\psi_2(k_m))}$ and thus that a sequence $(l_m)_{m\ge 0}$ satisfying \eqref{eq2} can be constructed.

Note that the sequences $(j_{m})_{m\ge 0}$, $(\psi_2(k_m))_{m\ge 0}$ and $(k_m)_{m\ge 0}$ tends to infinity as $m$ tends to infinity since $\psi_2(k_m)> \psi_1(k_m)> l_{m-1}$ and $(l_m)_{m\ge 0}$ is increasing. Moreover, we have for every $m\ge 1$,
\[ X_{l_0}\le \big(\prod_{s=1}^{m}2^{l_s}\gamma^{1/2}_{\psi_2(k_s)}\big)X_{l_{m}}\le \gamma^{1/2}_{\psi_2(k_1)}2^{l_m}X_{l_m}\le 2^{l_m}X_{l_m}.\] 
\par\smallskip

For every $j\ge 0$ and $n\ge 0$, since 
\[
\smash[b]{\|T^{\,j}\,x\|\ge\|P_{n}\,T^{\,j}\,x\|\ge 
\|P_{n}\,T^{\,j}P_{n}\,x\|-
\sum_{l>n}\|P_{n}\,T^{\,j}P_{l}\,x\|},
\]
we have by definition of $j_m$ 
\[
\|T^{\,j}x\|\ge \|P_{l_{m}}T^{\,j}P_{l_{m}}\,x\|-X_{l_{m}}\quad  
\textrm{for every}\ 0\le j< j_{m}.
\]
It follows from this inequality  that for every $m\ge 1$,
\begin{align*}
& \bigl\{0\le j\le j_{m}-1\,:\,\|P_{l_{m}}T^{\,j}P_{l_{m}}\,x\|
\ge 2^{2l_{m}}X_{l_{m}}\bigr\}\\
&\quad\subseteq
 \bigl\{0\le j\le j_{m}-1\,:\,\|P_{l_{m}}T^{\,j}P_{l_{m}}\,x\|
\ge (2^{l_{m}}+1)X_{l_{m}}\bigr\}\\
&\quad\subseteq \bigl\{0\le j\le j_{m}-1\,:\,\|T^{\,j}x\|\ge 2^{l_{m}}X_{l_{m}}\bigr\}\\
&\quad\subseteq
\bigl\{0\le j\le j_{m}-1\,:\,\|T^{\,j}x\|\ge X_{l_{0}}\bigr\}.
\end{align*}
Hence, we have
\[
 \overline{\textrm{dens}}\ N\bigl(x,B(0,X_{l_0} )^c\bigr)\ge 
\limsup_{m}\ \dfrac{\ \#\bigl\{0\le 
j\le 
j_m -1\,:\, \|P_{l_{m}}\,T^{\,j}P_{l_{m}}\,x\|\ge 
2^{2l_{m}}X_{l_{m}}\bigr\}}{j_m}\cdot 
\]

Finally, by applying Proposition~\ref{Proposition 51} with $K_0=k_m\delta^{\psi_2(k_m)}+2l_m+1$, $K_1=\Delta^{(k_m)}-3\delta^{(\psi_2(k_m))}-2l_m-1$ and thus $\alpha=2^{-(2l_m+1)}$, we get for every $J\ge 0$,
\begin{align*}
&\dfrac{\ \#\bigl\{0\le 
j\le 
J\,:\, \|P_{l_{m}}\,T^{\,j}P_{l_{m}}\,x\|\ge 
2^{2l_{m}}X_{l_{m}}\bigr\}}{J+1}\\
&\quad\quad\quad \ge 1-2\Big((k_m+3)\delta^{(\psi_2(k_m))}+4l_m+2\Big)\Bigl(\dfrac{1}{J+1}+\dfrac{1}{\Delta^{(k_m)}} \Bigr).
\end{align*}
We deduce that
\begin{align*}
&\overline{\textrm{dens}}\ N
\bigl(x, B(0,X_{l_0} )^c\bigr)\\
&\ \ge \limsup_m
\Bigl[1-2\Big((k_m+3)\delta^{(\psi_2(k_{m}))}+4l_{m}+2\Big)\Bigl(\dfrac{1}{j_{m}}+\dfrac{1}{\Delta^{(k_{m})}} \Bigr)\Bigr].
\end{align*}
Therefore, since  $j_{m}>\delta^{(\psi_2(k_{m+1}))}$, since $\psi_2(k_{m+1})>\psi_1(k_{m+1})>l_{m}\ge n_{k_{m}}\ge k_{m}$, since $\psi_2(k_{m+1})>\psi_2(k_{m})$ and since $l_m<n_{k_{m}+1}$, we can conclude from $\lim_{k\to\infty }\dfrac{k\delta ^{(k)}}{\delta 
^{(k+1)}}=0$ and $\lim_{k\to\infty }\ \dfrac{k\delta ^{(\psi_2(k))}+n_{k+1}}{\Delta 
^{(k)}}=0$ that 
\begin{align*}
&\overline{\textrm{dens}}\ N
\bigl(x, B(0,X_{l_0})^c\bigr)\\
&\ \ge \limsup_m\left[
1-\dfrac{2(k_m+3)\delta^{(\psi_2(k_{m}))}}{\delta^{(\psi_2(k_{m+1}))}}
-\dfrac{8(\psi_2(k_{m+1})-1)}{\delta^{(\psi_2(k_{m+1}))}}\right.\\
&\quad\quad\quad\quad
\left.-\dfrac{4}{j_{m}}
-\dfrac{2(k_m+3)\delta^{(\psi_2(k_{m}))}}{\Delta^{(k_{m})}}-\dfrac{8n_{k_{m}+1}}{\Delta^{(k_{m})}}-\dfrac{4}{\Delta^{(k_{m})}}\right]\\
&= 1
\end{align*}
The vector $x$ is thus not frequently hypercyclic since $\underline{\textrm{dens}}\ N\bigl(x, B(0,X_{l_0} )\bigr)=1-\overline{\textrm{dens}}\ N\bigl(x, B(0,X_{l_0} )^c\bigr)=0$ and $X_{l_0}>0$.
\end{proof}

We can now state and prove the main result of this section.

\begin{theorem}
There exists an operator on $\ell_1(\mathbb{N})$ which is $UFHC_a$ for every $a\in \mathcal{A}$ and not frequently hypercyclic.
\end{theorem}
\begin{proof}
Let $\psi(k)=(\psi_1(k),\psi_2(k))$ such that for every $k\ge 1$, 
\[
1\le \psi_1(k)< \min\{k+1,\psi_2(k)\}\quad\text{and}\quad \psi_2(k)> \max\{\psi_2(j):j< \psi_1(k)\}
\] and such that for every $i\ge 1$, there exists $j_i$ such that for every $j\ge j_i$, the set $\{k:\psi(k)=(i,j)\}$ is infinite. Let $n_0=0$, $n_1=1$ and $n_{k+1}=n_k+\psi_1(k)$ for every $k\ge 1$. For every $k\ge 1$, every $n\in [n_k,n_{k+1})$, we then let $\varphi(n)=n-n_k$, $
v_n=v^{(k)}=2^{-\tau^{(\psi_2(k))}}$, $b_{n+1}-b_n=\Delta^{(k)}$ and for every $j\in (b_n,b_{n+1}-1]$,
\[
w_j=
\begin{cases}
  2 & \quad\text{if}\ \ b_n< j\le b_n+k\delta^{(\psi_2(k))}+2n+1\\
 1 & \quad\text{if}\ \ b_n+k\delta^{(\psi_2(k))}+2n+1< j<\Delta^{(k)}-3\delta^{(\psi_2(k))}-2n-1\\
 1/2 & \quad\text{if}\ \  \Delta^{(k)}-3\delta^{(\psi_2(k))}-2n-1\le i 
<\Delta^{(k)}-2\delta^{(\psi_2(k))}\\
 2 & \quad\text{if}\ \ \Delta^{(k)}-2\delta^{(\psi_2(k))}\le i< \Delta^{(k)}-\delta^{(\psi_2(k))}\\
 1& \quad\text{if}\ \ \Delta^{(k)}-\delta^{(\psi_2(k))}\le i<\Delta^{(k)}
\end{cases}
\]
where $(k+3)\delta^{(\psi_2(k))}+4n_{k+1}+2<\Delta^{(k)}$ for every $k\ge 1$, $(\delta^{(k)})$ and $(\tau^{(k)})$ are increasing sequence of positive integers and for every $k\ge 1$, $\Delta^{(k)}$ is a multiple of $2\Delta^{(k-1)}$ with $\Delta^{(0)}=b_1-b_0$.
Let $\gamma_k:=2^{k\delta^{(k-1)}+2n_{k}+1-\tau^{(k)}}$.
Suppose that 
\begin{enumerate}
\item $\lim_{k\to \infty} \delta^{(k)}-\tau^{(k)}=\infty$;
\item for every $k\ge 1$, $\gamma_k\le \min\{2^{-16k}, \min_{s\le k-1}\gamma^2_{s}\}$;
\item $\lim_{k\to\infty }\dfrac{k\delta ^{(k)}}{\delta^{(k+1)}}=0$;
\item  $\lim_{k\to\infty }\ \dfrac{k\delta ^{(\psi_2(k))}+n_{k+1}}{\Delta 
^{(k)}}=0$.
\end{enumerate}
Then the operator of C-type $T_{v,w,\varphi,b}$ is $UFHC_a$ for every $a\in \mathcal{A}$ by Theorem~\ref{UFHC} but not frequently hypercyclic by Theorem~\ref{FHC}, and it suffices to remark that Conditions $(1)-(3)$ can be simultaneously satisfied by chosing by induction
\begin{itemize}
\item $\tau^{(k)}$ sufficiently big so that $\gamma_k\le \min\{2^{-16k}, \min_{s\le k-1}\gamma^2_{s}\}$, and then
\item $\delta^{(k)}$ sufficiently big so that $\delta^{(k)}-\tau^{(k)}\ge k$ and $\frac{(k-1)\delta^{(k-1)}}{\delta^{(k)}}\le \frac{1}{k}$,
 \end{itemize}
and that when the sequence $(\delta^{(k)})_k$ has been completely fixed, the remaining conditions can be satisfied by chosing for every $k\ge 1$
\begin{itemize}
\item $\Delta^{(k)}$ sufficiently big so that $\Delta^{(k)}>(k+3)\delta^{(\psi_2(k))}+4n_{k+1}+2$, $\Delta^{(k)}$ is a multiple of $2\Delta^{(k-1)}$ and $\dfrac{k\delta^{(\psi_2(k))}+n_{k+1}}{\Delta^{(k)}}\le \frac{1}{k}$.
\end{itemize}
\end{proof}

\end{document}